\documentclass[12pt,reqno]{amsart}
\usepackage{amsfonts}
\usepackage{latexsym,amsmath,amsthm,amssymb,amsfonts}
\usepackage{fullpage, setspace}
\usepackage{galois}
\numberwithin{equation}{section}
\allowdisplaybreaks
\def\endproof{$\hfill\Box$\\}
\def\s{\,\,\,\,}
\def\R{\mathbb{R}}
\def\C{\mathbb{C}}

\def\N{\mathbb{N}}
\def\Z{\mathbb{Z}}

\newcommand{\me}{\mathrm{e}}
\newcommand{\mi}{\mathrm{i}}
\numberwithin{equation}{section}
\newtheorem{theorem}{Theorem}[section]
\newtheorem{lem}[theorem]{Lemma}
\newtheorem{thm}[theorem]{Theorem}
\newtheorem{pro}[theorem]{Proposition}
\newtheorem{cor}[theorem]{Corollary}
\newtheorem{defi}[theorem]{Definition}
\newtheorem{rem}[theorem]{Remark}

\def\diam{\mbox{diam}}
\def\injrad{\mbox{injrad}}
\newcounter{Cnumber}

\title[ ]
{\bf \bf  Weak Convergence Of Branched Conformal Immersions With Uniformly Bounded Areas And Willmore Energies}
\author[ ]{Guodong Wei}

\date{}
\begin{document}
\maketitle
\begin{abstract}
In this paper, we firstly extend Theorem 5.1.1 in \cite {Helein} due to H\'elein to a rescaled branched conformal immersed sequence(c.f. Theorem 1.5). By virtue of this local convergence theorem, we study the blowup behavior of a sequence of branched conformal immersions of closed Riemannian surface in $\mathbb{R}^{n}$ with uniformly bounded areas and Willmore energies. Furthermore, we prove that the integral identity of Gauss curvature is true.
\end{abstract}

\section{Introduction}

The study of the conformal immersions from a Riemannian manifold into another Riemannian manifold has been an important topic in conformal geometry and has a long history. There are a great many mathematicians devoting themselves to this area (c.f. \cite{bryant, Chen, Chen-Li, Kuwert-Schatzle, lamm-scha, moore, Muller- sverak, riviere2} and references therein). Especially, some mathematicians employed the so-called $W^{2,2}$ conformal immersions of surfaces to study the minimizer of Willmore functional in the last ten years(c.f. \cite{Chen-Li, Kuwert-Li, Kuwert-scha, Riviere, riviere1}). A basic tool which they adopted to study this problem is the following local convergence theorem (c.f. Theorem 5.1.1 in \cite {Helein}) which was established by F. H\'elein and has played an important role in the analysis of 2-dimensional conformally invariant problems.

Let
 $$\mathcal{C}_\gamma(D)=\left\{f\in C^{\infty}(\bar{D},\R^{n})| f\ \mbox{is a conformal immersion, and}\  \int_D|A_f|^{2}d\mu_f\leq\gamma\right\}.$$
Here $A_f$ denotes the second fundamental form of the conformal immersion $f:D\rightarrow {\R}^n$, where $D$ is a disc in $\R^2$. H\'elein's theorem can be stated as follows.
\begin{thm}[\cite{Helein}] \label{helein}
Let $\overline{\mathcal{C}_\gamma(D)}$ be the closure of $\mathcal{C}_\gamma(D)$ with respect to the weak topology of $H^{1}(B,\R^{n})$. Then, for all $\gamma\leq\frac{8\pi}{3}$, every map $f\in\overline{\mathcal{C}_\gamma(D)}$ is either a constant map or a bi-Lipschitz conformal immersion.
\end{thm}

The main motivation of the present paper is to study how does the topology change in the process of blowing up for a sequence of branched conformal immersions of a closed Riemann surface into $\mathbb{R}^{n}$ with uniformly bounded area and Willmore energy. Before stating the main result, we need to introduce some notions.

For an immersion $f:\Sigma\rightarrow \mathbb{R}^{n}$  of a compact Riemannian surface, we denote by $g_{f}=f^{*}g_{euc}=df\otimes df$ the pull-back metric, $\mu_{f}$ the induced area measure on $\Sigma$ and $\Delta_{g_{f}}$ the Laplace-Beltrami operator. It is well known that
\begin{equation}
\Delta_{g_{f}}f=\overrightarrow{H_{f}},
\end{equation}
where $\overrightarrow{H_{f}}$ is the mean curvature vector corresponding to the immersion $f$. The Willmore functional (or Willmore energy) of $f$ is defined by
\begin{equation*}
W(f)=\frac{1}{4}\int_{\Sigma}|H_{f}|^{2}d\mu_{f}.
\end{equation*}
By the Gauss equation and the Gauss-Bonnet theorem,  the Willmore functional can be also written as
\begin{equation}\label{willmore}
W(f)=\frac{1}{4}\int_{\Sigma}|A_{f}|^{2}d\mu_{f}+\pi\chi(\Sigma),
\end{equation}
where $\chi(\Sigma)$ is the Euler characteristic of $\Sigma$. It is also well known that the Willmore functional $W(f)$ is of the invariance under the conformal transformations of $\R^{n}$ (c.f. \cite{Chen, Willmore}).

Next, we recall another two important notions: $W^{2,2}$ conformal immersion and $W^{2,2}$ branched conformal immersion.
\begin{defi}
Let $\left(\Sigma,g\right)$ be a Riemann surface. A map $f\in W^{2,2}\left((\Sigma,g), \mathbb{R}^{n}\right)$ is called a $W^{2,2}$ conformal immersion if the induced metric $g_{f}=df\otimes df$ is given by
$$g_{f}=\me^{2u}g\s\s\mbox{with}\s u\in\ L^{\infty}(\Sigma).$$
We say $f\in W^{2,2}_{conf,loc}\left((\Sigma,g), \mathbb{R}^{n}\right)$ if $f\in W^{2,2}_{loc}\left((\Sigma,g), \mathbb{R}^{n}\right)$ and $u\in L^{\infty}_{loc}(\Sigma)$.

Moreover, a map $ f:\Sigma\rightarrow\mathbb{R}^{n} $ is called a $W^{2,2}$ branched conformal immersion with possible branch points $x_{1}, \cdots, x_{m}$ if $ f \in\ W^{2,2}_{conf,loc}(\Sigma\backslash \{x_{1}, \cdots, x_{m}\},\mathbb{R}^{n})$  with
 $$\int_{\Sigma\backslash \{x_{1}, \cdots, x_{m}\}}(1+|A_{f}|^{2})d\mu_{f}<\infty .$$
 \end{defi}

Given a Riemann surface $(\Sigma, g)$, we denote the set of all $W^{2,2}$ conformal immersions  and all $W^{2,2}$ branched conformal immersions of $(\Sigma, g)$ by $W^{2,2}_{conf}\left((\Sigma,g), \mathbb{R}^{n}\right)$ and $W^{2,2}_{b, c}((\Sigma, g), \R^{n})$ respectively. Obviously, $W^{2,2}_{conf}((\Sigma, g), \mathbb{R}^{n})$ and $W^{2,2}_{b, c}((\Sigma, g), \R^{n})$ depend only on the conformal class of $(\Sigma, g)$, not on the choice of $g$.
 \begin{rem}
We note that it follows from Theorem \ref{remove} (see it in Section 2) that every possible branch point $p\in \Sigma$ of a $W^{2,2}$ branched conformal immersion $f\in W^{2,2}_{b, c}((\Sigma, g), \R^{n})$ has a well-defined branching order $m(p)\in \N_{0}$.
 \end{rem}
Here we make a convention that $\Sigma_{k}=(\Sigma,h_{k})$ denotes a Riemann surface which is equipped a smooth metric $h_{k}$ with constant Gauss curvature $\pm1$, or $(\Sigma,h_{k})$ is just a torus of genus 1 with Gauss curvature $0$.

One pays attention to conformal immersions because of the conformal immersions are of two advantages. One is that the conformal diffeomorphism group of $(\Sigma,h_k)$ is rather small comparing with the group of diffeomorphisms. The other is: if $g_{f}=e^{2u}g_{euc}$ in an isothermal coordinate chart, then one can estimate $\|u\|_{L^{\infty}}$ from the compensated compactness property of $K_{f}e^{2u}$. Thus it is possible to get an upper bound of $\|f\|_{W^{2,2}}$ via the equation $\Delta_{g_{f}}f=H_{f}$.

Here we want to list some local convergence results. When $m=0$, the following theorem is due to H\'{e}lein with $\gamma_n=\frac{8\pi}{3}$ (c.f. Theorem \ref{helein}) and later, Kuwert and Li proved $\gamma_n$ is the optimal constant(c.f. Theorem 5.1 in \cite{Kuwert-Li}); In the case $m\geq1$, the theorem was proved by Lamm and Nguyen (c.f. Theorem 3.2 in \cite{lamm-nguyen}).
\begin{thm}[\cite{lamm-nguyen}]\label{thm1}
Let $f_{k} \in  W^{2,2}_{conf,loc}(D\backslash\{0\},\ \mathbb{R}^{n})$ be a sequence of branched conformal immersions with branching order m at $0$. Denote the induced metric of $f_k$ by $g_{f_k}=e^{2u_{k}}g_{euc}$ (i.e. $u_{k}(z)=m\log|z|+\omega_{k}(z)$). Assume that $ \mu_{f_{k}}(D)\leq C$, $ f_{k}(0)=0 $ and
$$\int_{D}|A_{f_{k}}|^{2}d\mu_{f_{k}}\leq\gamma<\gamma_{n}=\left\{
\begin{aligned}
8\pi\s\s \mbox{if} \s n=3,\\
4\pi\s\s\mbox{if}\s n\geq4.
\end{aligned}
\right.
$$
Then $f_{k}$ is also bounded in $W^{2,2}_{loc}(D,\mathbb{R}^{n})$ and there is a subsequence such that one of the following two alternatives holds:
\begin{enumerate}
\item[(1)] $f_{k}$ converge weakly in $W^{2,2}_{loc}(D,\mathbb{R}^{n})$ to a conformal immersion $f$ with branching order $m$ at $0$, i.e.,\ $g_{f}=|z|^{2m}e^{\omega(z)}\delta_{ij}$, moreover,  $$\|\omega_{k}\|_{L^{\infty}(D_{r}(0))}+\|D\omega_{k}\|_{L^{2}(D_r(0))}\leq C(r).$$
\item[(2)] $\omega_{k}\rightarrow -\infty$ and $f_{k}\rightarrow 0$ locally uniformly on $D$.
\end{enumerate}
\end{thm}

Here, we should mention that Kuwert and Li \cite{Kuwert-Li} discussed the compactness for a sequence of conformal immersions of a compact Riemann surface. More precisely, they proved that, if $f_{k}\in W^{2,2}_{conf}(\Sigma_{k}, \mathbb{R}^{n})$ are conformal immersions with $W(f_k)<\Lambda<\infty$, and $\Sigma_{k}$ converge to $\Sigma$ in moduli space, then there exist M\"{o}bius transformations $\sigma_{k}$ and diffeomorphisms $\phi_k$, such that $\sigma_{k}\comp f_{k}\comp \phi_k$ converge to $f_{0}$ locally in weak $W^{2,2}$ sense on $\Sigma$ minus finitely many concentration points and the weak limit $f_{0}$ is a $W^{2,2}$ branched conformal immersion. Moreover, $f_0$ is unbranched and an embedding if $\Lambda<8\pi$.  Rivi\'ere \cite{riviere1} also proved the same results. With the additional condition $\mu_{f_{k}}(\Sigma_{k})\leq C$, furthermore Chen and Li \cite{Chen-Li} studied the Hausdorff limit of $f_{k}(\Sigma)$ and proved the area identity.

\medskip

Eschenburg and Tribuzy \cite{Eschenburg-Tribuzy} have ever studied the Gauss-Bonnet formula for branched conformal immersions. Following the work \cite{Kuwert-Li}, \cite{riviere1} and \cite{Chen-Li}, naturally, one wants to know whether the integral of Gauss curvature is maintained after blowing up? For the present problem, we can not neglect the bubbles which shrink to a point. This is quite different from the case in \cite {Chen-Li}, there the bubbles shrinking to a point can be directly ignored when one discussed the Horsdorff limit and area identity. The main reason is that Theorem \ref{helein} and Theorem \ref{thm1} does not guarantee the limit is an immersion. For example, let $S^{2}$ be the standard unit two-sphere in $\R^{3}$ and then define a sequence of conformal immersions $f_k:S^{2}\rightarrow \R^{3}$ by $f_k(x)=\frac{x}{k}$, then for each $k$, it is easy to check that
 \begin{equation}
 \int_{S^{2}}K_{f_{k}}d\mu_{f_k}=4\pi\s\s\s\s \mbox{and}\s\s\s\s  \int_{S^{2}}|A_{f_{k}}|^{2}d\mu_{f_k}=8\pi.
 \end{equation}
However, the limit of $f_k$ is a point map.

Based on the above consideration we need to study those bubbles which carry nontrivial topology but shrink to a point. In order to recover the lost curvature integral at these bubbles, a natural attempt is to rescale $f_k$ by an appropriate constant to make the limit nontrivial. Fortunately, by employing the compensated compactness property of $K_{f_k}e^{2u_k}$ (c.f. \cite {Muller- sverak, Helein, Kuwert-Li, riviere1}), Simon's monotonicity formula (c.f. \cite{Simon}) and Theorem \ref{thm1} we obtain the following theorem which can be viewed as an extension of H\'elein's local convergence theorem, i.e., Theorem \ref{helein}.

\begin{thm}\label{thm2}
Let $f_{k}: D\longrightarrow\mathbb{R}^{n}$ be a sequence of $W^{2,2}$ branched conformal immersions with branching order m at $0$ which satisfies
\begin{enumerate}
\item[(1)] $$\int_{D}|A_{f_{k}}|^{2}d\mu_{f_{k}}\leq\gamma<\gamma_{n}=\left\{
\begin{aligned}
8\pi\s\s\mbox{ if }\s n=3,\\
4\pi\s\s\mbox{ if }\s n\geq4.
\end{aligned}
\right.$$
\item[(2)] $f_{k}(D)$  can be extended to a closed branched immersed surface $f_{k}:\Sigma\rightarrow \R^{n}$ with
$$\int_{\Sigma}|A_{f_{k}}|^{2}d\mu_{f_{k}}<\Lambda.$$
\end{enumerate}
Take a curve $\gamma:[0,1]\longrightarrow D$ and set $\lambda_{k}=\mbox{diam}f_{k}(\gamma)$, then we can find a subsequence of $\{f_{k}'=\frac{f_{k}-f_{k}(\gamma(0))}{\lambda_{k}}\}$ which converges weakly in $W^{2,2}_{loc}(D)$ to a $W^{2,2}$ branched conformal immersion $f_{0}$ with branching order m at $0$. Moreover, there holds true
$$\|\omega_{k}'\|_{L^{\infty}(D_r(0))}\leq C(r),$$
where $\omega_{k}'$ satisfies
$$df_{k}'\otimes df_{k}'=e^{2(\omega_{k}'+m\log|z|)}(dx^{1}\otimes dx^{1}+dx^{2}\otimes dx^{2}).$$
\end{thm}

Combining Theorem \ref{thm1} with Theorem \ref{thm2}, we can give a bubble tree construction for $W^{2,2}$ conformal branched immersions with uniformly bounded Willmore energies and areas. Moreover, we obtain the following integral identity of Gauss curvature .
\begin{thm}\label{identity}
Suppose that $\{f_{k}\}$ is a sequence of $W^{2,2}$ branched conformal immersions of closed Riemann surface $(\Sigma, g)$ in $\R^{n}$ which have the same branch points and same branching order at each branch point. Assume $f_k(\Sigma)\cap B_{R_{0}}(0)\neq\emptyset$ for some $R_0>0$ and
$$\sup_{k}\{\mu(f_{k})+W(f_{k})\}<+\infty.$$
Then, away from a finite set, $f_{k}$ will converge either to some $\widetilde{f}_{0}\in W^{2,2}_{b,c}((\Sigma,g),\R^{n})$ or to a point. Moreover, we obtain the following Gauss curvature identity:
$$2\pi\chi(\Sigma)+2\pi b=\int_{\Sigma}K_{f_{0}}d\mu_{f_{0}}+\sum_{i}\int_{\C}K_{f^{i}}d\mu_{f^{i}},$$
where $b$ is the sum of the branching order of all branch points, $f_0$ is the generalized limit of $f_{k}$ (c.f. Definition \ref{general} in section $3$) and $f^{1}, \cdots, f^{m}$ are bubbles.
\end{thm}

Next, we consider the convergence behavior of a sequence of branched conformal immersions of $\Sigma_k$ with varied complex structures and varied branch points. Let $\{f_{k}\}$ be a sequence of $W^{2,2}$ branched conformal immersions from $\Sigma_{k}$ with branch points $\{x_{k,1},\cdots,x_{k,m}\}$. If $\{f_{k}\}$ has the same branching order at $x_{k,j}$ for $j=1, \cdots, m$ and $k\in \mathbb{N}^{+}$, and
$$ (\Sigma_{k},x_{k,1},\cdots, x_{k,m})\rightarrow (\Sigma_{0},x_{0,1},\cdots, x_{0,m})\s\s \mbox{in}\s \overline{\mathcal{M}}_{g,m},$$
then we have the following theorem.

\begin{thm}\label{identity2}
With the notations as above, assume $f_k(\Sigma)\cap B_{R_{0}}(0)\neq\emptyset$ and
$$\sup_{k}\{\mu(f_{k})+W(f_{k})\}<+\infty.$$
Then we have
$$2\pi\chi(\Sigma)+2\pi b=\sum_{i}^{s}\int_{{\Sigma_{0}^{i}}}K_{f_{0}^{i}}d\mu_{f_{0}^{i}}+\sum_{j}\int_{\C}K_{f^{j}}d\mu_{f^{j}},$$
where $\Sigma_{0}^{i}$ is the component of $\Sigma_{0}$ minus nodal points, $b$ is the sum of the branching order of the corresponding branch points, $f_{0}^{i}$ is the generalized limit of $f_k\comp \psi_k$ on $\Sigma_{0}^{i}$  where $\psi_k$ is a map from $\Sigma_{0}^{i}$ to $(\Sigma,h_k)$,  and $f^{1}, \cdots, f^{m}$ are bubbles.
\end{thm}

\begin{rem}
By the theory of  muduli space of Riemannian surface with marked points (for example, see proposition 5.1 on page 71 in \cite{Hummel} ), for each sufficiently large $k$, there exists a diffeomorphism $\psi_{k}:\Sigma_{0}\backslash\mathcal{N}\rightarrow\Sigma_{k}\backslash\Gamma_{k}$, where $\mathcal{N}$ denotes the nadal points on $\Sigma_0$, $\Gamma_{k}$ is a set constituting pairwise disjoint, simple closed curves in $(\Sigma,h_k)$. When we consider the convergence behavior of the integral of Gauss curvature of a sequence of branched conformal immersions of $(\Sigma,h_k)$ with varied complex structures and varied branch points, we firstly consider the convergence behavior of $f_k\comp \psi_k$ in $\Sigma_{0}\backslash\mathcal{N}$. Near the nodal point, we use Collar lemma (c.f. proposition 5.1 on page 71 in \cite{Hummel}) to study the convergence behavior $f_k$. We refer the reader to section 5 to see the detailed proof of theorem 1.7.
\end{rem}

The paper is organized as follows. In Section \ref{preliminary}, we recall some preliminary facts on $W^{2,2}$ conformal branched immersion. In Section \ref{singular}, we give the proof of Theorem \ref{thm2}. In Section \ref{bubble}, we utilize Theorem \ref{thm1} and \ref{thm2}, to give a bubble tree construction and prove Theorem \ref{identity}. In Section 5, we provide the proof of Theorem \ref{identity2}.

\section{Preliminary}\label{preliminary}
In this section, we will recall some previous works which will be used in our paper.

Given a map $f\in W^{2,2}_{b,c}((\Sigma,g),\mathbb{R}^{n})$. Then, for any $p\in\Sigma$, by Theorem \ref{remove} we may choose a isothermal coordinate system around $p$ such that the induced metric $g_{ij}=\langle\partial_{i}f,\partial_{j}f\rangle$ can be given by
$$g_{ij}=e^{2u}g_{ij}=|z|^{2m}e^{2\omega}\delta_{ij}\s\s\mbox{where}\s \omega\in C^{0}\cap W^{1,2}_{loc}(U),$$
where $m$ is the branching order of $p$ ($p$ is unbranched iff $m=0$). Moreover, $u$ satisfies the following weak Liouville type equation
$$-\Delta u\ =K_{f}\me^{2u}-2m\pi\delta_{p}.$$
By assumption $K_{f}\me^{2u}$ belongs to $L^{1}$, it seems that the classical elliptic equation theory can not be used here. In \cite{Muller- sverak} M\"{u}ller and \^{S}ver\'{a}k discovered that this term can be written as a sum of Jacobi determinants, thus they found that the function $K_{g}e^{2u}$ actually belongs to the Hardy space $\mathcal{H}^{1}$ by \cite{Coifman}.  Kuwert and Li in \cite{Kuwert-Li} summed up the above as the following result.
\begin{thm}[Corollary 2.1 in \cite{Kuwert-Li}]\label{solution}
For $f\in W^{2,2}_{conf}(D,\ \mathbb{R}^{n})$ with the induced metric $g_{ij}=e^{2u}\delta_{ij},$ assume
$$\int_{D}|A_{f}|^{2}d\mu_{f}\leq\gamma<\gamma_{n}=\left\{
\begin{aligned}
8\pi\s\s\mbox{ if } n=3,\\
4\pi\s\s\mbox{ if } n\geq4.
\end{aligned}
\right.$$
Then there exists a function $\upsilon:\mathbb{ C }\rightarrow\mathbb{ R }$ solving the following equation
$$-\Delta \upsilon=K_{g}e^{2u}\s\s \mbox{in}\s D,$$
and satisfying the estimates
$$\|\upsilon\|_{L^{\infty}(\C)}+\|D\upsilon\|_{L^{2}(\C)}\leq C(\gamma)\int_{D}|A_f|^{2}d\mu_{f}.$$
\end{thm}
\begin{rem}\label{userem}
With the same condition as above, since $u-v$ is a harmonic function, it is easy to see that
$$\|u\|_{L^{\infty}(D_r)}+\|D u\|_{L ^{2}(D_r)}\leq C(\gamma,r)(\gamma+\|u\|_{L^{1}(D)}).$$
\end{rem}
Let $\Omega^{*}_{r}=\{z\in \mathbb{C}:|z|>r\}$. A metric on $\Omega^{*}_{r}$ is complete at $\infty$ if there exists some $z_{0}\in\Omega^{*}_{r}$ such that $dist_{g}(z_{0},z)\rightarrow\infty$ when $z\rightarrow\infty$.  M\"{u}ller and \^{S}ver\'{a}k \cite{Muller- sverak} studied this type surface and obtained the following theorem:
\begin{thm}[\cite{Muller- sverak}, Theorem 4.2.1]\label{infty}
Suppose that $f\in W^{2,2}_{conf,loc}(\mathbb{C}\backslash D_{R},\mathbb{R}^{n})$ with
$$\int_{\mathbb{C}\backslash D_{R}}|A_{f}|^{2}d\mu_{f}<\infty,$$
where $g_{ij}=e^{2u}\delta_{ij}$ is the induced metric. If $f(\mathbb{C}\backslash D_{R})$ is complete, then
$$ u(z)=m\log|z|+\omega(z),$$
where $m\in \mathbb{N}_{0}$, $\omega \in C^{0}\cap W^{1,2}(\mathbb{C}\backslash D_{R},\, \mathbb{R}^{n})$.
Moreover
$$\lim_{t\rightarrow \infty}\int_{\partial D_{t}}\frac{\partial u}{\partial r}=2m\pi.$$
\end{thm}

For a conformal immersion from a punctured disk,  Kuwert and Li \cite{Kuwert-Li} proved the following removable singularity result.
\begin{thm}[\cite{Kuwert-Li}]\label{remove}
 Suppose that $f\in W^{2,2}_{loc}(D\backslash\{0\},\ \mathbb{R}^{n})$ satisfies$$\int_{D\backslash\{0\}}|A_{f}|^{2}d\mu_{f}<\infty\s\mbox{and}\s\mu_{f}(D\backslash\{0\})<\infty,$$ where $g_{f}=e^{2u}g_{euc}$ is the induced metric. Then $f\in W^{2,2}(D,\ \mathbb{R}^{n})$ and
$$ u(z)=m\log|z|+\omega(z)\ \ \  \mbox{where} \ m\in \mathbb{N}_{0},\ \  \omega \in C^{0}\cap W^{1,2}(D,\ \mathbb{R}^{n}),$$
$$-\Delta u\ =\ -2m\pi\delta_{0}+K_{g}e^{2u}\ \ \ \mbox{in}\s  D.$$
Moreover, the multiplicity of $f(D_\delta(0))$ as varifold at $f(0)$ is $m+1$.
\end{thm}
The number $m$ in the above Theorem \ref{remove} is called the branching order at $0$. Obviously, $f$ is unbranched at $0$ if and  only if $m=0$.
\begin{rem}\label{rmk}
It is worthwhile pointing out that if $f$ satisfies the same conditions as Theorem \ref{remove},  then  the following two useful facts hold:
\begin{enumerate}
\item[(1)] $\lim\limits_{z\rightarrow 0}\frac{|f(z)-f(0)|}{|z|^{m+1}}=\frac{e^{\omega(0)}}{m+1}.$
\item[(2)] Integrating $ -\Delta u\ =\ -2m\pi\delta_{0}+K_{g}\me^{2u}$ on $D_t$ , we get
 $$\lim_{t\rightarrow 0}\int_{\partial D_{t}}\frac{\partial u}{\partial r}=2m\pi.$$
\end{enumerate}
\end{rem}
At the end of this section, we recall the Gauss-Bonnet formula for conformal immersion with branch points.
\begin{thm}[\cite{Chen-Li,Eschenburg-Tribuzy}]\label{gauss} 
Let $(\Sigma, g)$ be a closed Riemann surface. Then for any $f\in W^{2,2}_{b,c}((\Sigma, g), \mathbb{R}^{n})$, there holds
$$\int_{\Sigma}K_{f}d\mu_{f}=2\pi\chi(\Sigma)+2\pi b,$$
here $b$ is the sum of the branching order of the branch points.
\end{thm}

\section{Generalized H\'elein's convergence theorem}\label{singular}
In this section, we will employing Theorem \ref{thm1}, Simon's monotonicity formula (c.f. (1.3) in \cite{Simon}) and compensated compactness property of $K_{f_k}e^{2u_k}$ to provide the proof of Theorem \ref {thm2}. We also utilize Theorem \ref{thm1} and \ref{thm2} to give a convergence result for conformal minimal branched immersions at the end of this section.

\subsection{{\bf Proof of Theorem \ref{thm2}}:}
Firsty, note that $f_{k}'(\gamma)\subset B_{1}(0)$. Without loss of generality, we assume $\gamma(0)=0$. Put $$\Sigma_{k}'=\frac{f_{k}(\Sigma)-f_{k}(\gamma(0))}{\lambda_{k}}, $$
then we need to consider the following two cases:

Case $(1)$: $\mbox{diam}f_{k}'(D)< C$. For this case, by inequality $(1.3)$ in \cite{Simon} (for more general varifold case we refer to the appendix A in \cite{Kuwert-Schatzle} ) with $\rho=\infty$ we have that $$\frac{\mu\{\Sigma_{k}'\cap B_{\sigma}(0)\}}{\sigma^{2}}\leq C$$
for any $\sigma>0$. Hence we get $\mu_{f_{k}'}(D)<C$ by taking $\sigma=\mbox{diam}f_{k}'(D)$ . Then, it follows from Theorem \ref{thm1} that $f_{k}'$ converges weakly in $W^{2,2}_{loc}(D)$. Since $\mbox{diam}f_{k}'(\gamma)=1$, the weak limit is not trivial.
\medskip

Case $(2)$: $\mbox{diam}f_{k}'(D)\rightarrow\infty$. As shown in \cite{Kuwert-Li}(page 328), there exist a point $y_{0}\in\mathbb{R}^{n}$ and a constant $\delta>0$, such that
$$ B_{\delta}(y_{0})\cap\Sigma_{k}'=\emptyset,\s\s \mbox{for\s all} \s k.$$

Let $I=\frac{y-y_{0}}{|y-y_{0}|^{2}}$,
 $$f_{k}''=I\comp f_{k}'\s\s\mbox{and}\s\s \Sigma_{k}''=I(\Sigma_{k}').$$
By conformal invariance of Willmore functional, we have
$$\int_{\Sigma_{k}''}|A_{\Sigma_{k}''}|^{2}d\mu_{\Sigma_{k}''}=\int_{\Sigma_{k}}|A_{f_{k}}|^{2}d\mu_{f_{k}}<\Lambda.$$
Since $\Sigma_{k}''\subset B_{\frac{1}{\delta}}(0),$ again by $(1.3)$ in \cite{Simon}, we get $\mu_{f_{k}''}(\Sigma_k)<C $. Let
$$\mathcal{S}(\{f_{k}''\}):\ = \{p\in D:\lim_{t\rightarrow 0}\liminf_{k\rightarrow\infty}\int_{D_{t}(p)}|A_{f_{k}''}|^{2}d\mu_{f_{k}''}\geq\gamma_{n}\}.$$
Then $f_{k}''$ converges weakly in $W^{2,2}_{loc}(D\backslash\mathcal{S}(\{f_{k}''\})).$

Next, we prove that $f_{k}''$ does not converge to a point. If $f_{k}''$ converge to a point in $W^{2,2}_{loc}(D\backslash\mathcal{S}(\{f_{k}''\})$, then the limit must be $0$ which is implied by $\mbox{diam}(f_{k}')\rightarrow\infty$. However, by the definition of $f_{k}''$, we can find a $\delta_{0}$ such that $f_{k}''(\gamma)\cap B_{\delta_{0}}(0)=\emptyset$, thus for any $p\in \gamma([0,1])\backslash\mathcal{S}(\{f_{k}''\}),f_{k}''$ does not converge to a point. A contradiction.

It is remaining for us to prove that $f_{k}'$ converges weakly in $W^{2,2}_{loc}(D,\mathbb{R}^{n})$. Let $f_{0}''$ be the weak limit of $f_{k}''$ which is a branched immersion from $D$ to ${R}^{n} $ by Theorem \ref{remove}. Let $\mathcal{S}^{*}=f_{0}''^{-1}(0)$, by Remark \ref{rmk}, we know $\mathcal{S}^{*}$ is isolated.

First, we claim that, for any $\Omega\subset\subset D\backslash(\mathcal{S}^{*}\cup\mathcal{S}(\{f_{k}''\})$, $f_{k}'$ converges weakly in $W^{2,2}(\Omega,\mathbb{R}^{n})$.
Since $f_{0}''$ is continuous on $\bar{\Omega}$, we can assume $\mbox{dist}(0,f_{0}''(\Omega))>\delta>0$. Then $\mbox{dist}(0,f_{k}''(\Omega))>\frac{\delta}{2}>0$ when $k$ is sufficiently large. Noting that $f_{k}'=\frac{f_{k}''}{\|f_{k}''\|^{2}}+y_{0}$, we get $f_{k}'$ converges weakly in $W^{2,2}(\Omega,\mathbb{R}^{n})$.

If $0\in\mathcal{S}^{*}\cup \mathcal{S}(\{f_{k}''\})$, we know that $0\in \mathcal{S}(\{f_{k}''\})$ since $f_{k}'(0)=0$. Next, we want to prove that for each $p\in\mathcal{S}^{*}\cup \mathcal{S}(\{f_{k}''\})\backslash\{0\}$, $f_{k}'$ also converges in a neighborhood of $p$. Let $g_{f_{k}'}=e^{2u_{k'}}\delta_{ij}$. Since $f_{k}'\in W^{2,2}_{loc}(D_{4r}(p))$ with $\int_{D_{4r}(p)}|A_{f_{k}'}|^{2}d\mu_{f_{k}'}<\gamma_{n}$ when $r$ is sufficiently small and $k$ is sufficiently large, by Theorem \ref{solution} we know that there exists a $\upsilon_{k}$ which solves the following equation
$$-\Delta \upsilon_{k}=K_{f_{k'}}e^{2u_{k'}}\ \ \ z\in D_{4r}(p) \ \ \ \mbox{and}\ \ \|\upsilon_{k}\|_{L^{\infty}(D_{4r}(p))}\leq C.$$

Because of that $f_{k}'$ converges to a conformal immersion in $D_{4r}(p)\backslash D_{\frac{r}{4}}(p)$, by Theorem \ref{thm1} we may assume that
$$\|u_{k}'\|_{L^{\infty}(D_{2r} \backslash D_{r}(p))}\leq C.$$

Since $u_{k}'-\upsilon_{k}$ is a harmonic function with $\|u_{k}'-\upsilon_{k}\|_{L^{\infty}(\partial D_{2r}(p))}\leq C$, from the maximal principle we get $$\|u_{k}'-\upsilon_{k}\|_{L^{\infty}(D_{2r} (p))}\leq C.$$
Thus, the above inequality and the boundedness of $\upsilon_k$ lead to $\|u_{k}'\|_{L^{\infty}(D_{2r}(p))}\leq C$, which implies $$\|\nabla f_{k}'\|_{L^{\infty}(D_{2r} (p))}\leq C.$$
It follows that
$$\|e^{2u_{k}'}H_{f_{k}'}\|^{2}_{L^{2}(D_{2r} (p))}\leq e^{2\|u_{k}'\|_{L^{\infty}(D_{2r} (p))}}\int_{D_{2r} (p)}|H_{f_{k}'}|^{2}d\mu_{f_{k}'}\leq C.$$
Furthermore, from the equation $\Delta f_{k}'=e^{2u_{k}'}H_{f_{k}'}$ we infer
$\|f_{k}'\|_{W^{2,2}(D_{r} (p))}\leq C$. Thus a subsequence of $f_{k}'$ converges weakly.

If $0\in \mathcal{S}(\{f_{k}''\})$, by Theorem \ref{thm1},  we have that for sufficiently small $r$,
$$\|u_{k}'\|_{L^{\infty}(D_{2r} \backslash D_{r}(0))}\leq C,$$
since $f_{k}'$ converges to a conformal immersion in $D_{4r}\backslash D_{\frac{1}{4}r}(0)$.
Hence
$$\|\omega_{k}'\|_{L^{\infty}(D_{2r} \backslash D_{r}(0))}\leq C.$$

Noting that $\omega_k'-\upsilon_{k}$ is a harmonic function and using the above argument again, we conclude that there exists a subsequence of $f_{k}'$  which converges weakly in $W^{2,2}_{loc}$ to a branched conformal immersion in $D_{r}(0)$.
\endproof
\medskip

By Theorem \ref{thm1}, the following corollary holds obviously.

\begin{cor}
Let $f_{k}$ be a sequence of $W^{2,2}$ branched conformal immersions with bounded Willmore functional and induced area which has the same branch points and same branching order at each branch point. Denote
$$\mathcal{S}(\{f_{k}\})=\{p\in \Sigma:\lim_{t\rightarrow 0}\liminf_{k\rightarrow\infty}\int_{D_{t}(p)}|A_{f_{k}}|^{2}d\mu_{f_{k}}\geq\gamma_{n}\}.$$
If $\mathcal{S}(\{f_{k}\})=\emptyset$, then $f_k$ converges weakly in $W^{2,2}(\Sigma)$ either to some $\widetilde{f_0}\in W^{2,2}_{b, c}((\Sigma, g), \R^{n})$ or to a point.
\end{cor}

Now we give the definition of the generalized limit of $f_k$.

\begin{defi}\label{general}
Let $\{f_{k}\}$ be a sequence of $W^{2,2}$ branched conformal immersions with bounded Willmore functional and induced area which has the same branch points and same branching order at each branch point. Denote
$$\mathcal{S}(\{f_{k}\})=\{p\in \Sigma:\lim_{t\rightarrow 0}\liminf_{k\rightarrow\infty}\int_{D_{t}(p)}|A|_{f_{k}}^{2}d\mu_{f_{k}}\geq\gamma_{n}\}.$$
By Theorem \ref{thm1} and \ref{thm2}, we know there exists some point $x_{0}\in \Sigma\backslash\mathcal{S}(\{f_{k}\})$ and $\lambda_{k}=1\ or\ \lambda_{k}\rightarrow 0$ such that a subsequence $f_{k}'=\frac{f_{k}-f_{k}(x_{0})}{\lambda_k}$ converges in $W^{2,2}_{loc}(\Sigma\backslash \mathcal{S}(\{f_{k}\}))$ to a $W^{2,2}$  branched conformal immersion $f_{0}$ from $\Sigma\backslash \mathcal{S}(\{f_{k}\})$ to $\R^{n}$. Usually we call $f_{0}$ the generalized limit of $f_{k}$. We say the limit is trivial, if
$$\lambda_{k}\rightarrow 0\s\s  \mbox{and} \s\s  \int_{\Sigma\backslash \mathcal{S}(\{f_{k}\})} K_{f_{0}}d\mu_{f_0}=0.$$
\end{defi}
\begin{rem}
Here we want to give a remark about the generalized limit. For any $p\in \mathcal{S}\{f_k\}$, there are two possibilities for $\mu_{f_0}(B_{r}(p))$ where $B_r(p)$ is a small neighborhood of $p$. The one is $\mu_{f_0}(B_r(p))<\infty$, then by Theorem \ref{remove}, we know $p$ is a branch point of $f_0$;  Another possibility is $\mu_{f_0}(B_r(p))=\infty$, by Theorem \ref{infty}, we know $f_0$ is complete near $p$. In \cite{lamm-scha} and \cite{lamm-nguyen}, the authors defined the $W^{2,2}$ branched immersion by a different way (see definition 2.1 in \cite{lamm-scha} and definition 2.2 in \cite {lamm-nguyen}).
In the definition, branch points contains the point which satisfies either of the above two possibilities. So we can still view $f_0$ as a branched immersion of $\Sigma$ into $\R^{n}$. \end{rem}

\subsection{Convergence of conformal minimal branched immersions}
In this subsection, we will give an application of Theorem \ref{thm1} and Theorem \ref{thm2}. Let $N$ be a compact Riemannian manifold without boundary with dimension $m$. By Nash's imbedding Theorem, $N$ can be isometrically embedded in $\mathbb{R}^{n}$. So any immersion from $\Sigma$ in $N$ can be regarded as an immersion in $\mathbb{R}^{n}$.
\begin{pro}\label{minimal}
Let $f_{k}: D\rightarrow N$ be a sequence of conformal minimal branched immersions with branching order m at $0$ which satisfies
\begin{enumerate}
\item[(1)] $$\int_{D}|A_{f_{k}\comp D, \R^{n}}|^{2}d\mu_{f_{k}}\leq \gamma_{n}-\tau\s\s \mbox{and}\s\s \sup_{k}\mu_{f_{k}}(D)<\Lambda;$$
\item[(2)] $f_{k}(D)$  can be extended to a closed branched immersed surface $f_{k}:\Sigma\rightarrow \R^{n}$ with$$\int_{\Sigma}|A_{f_{k}\comp \Sigma, \R^{n}}|^{2}d\mu_{f_{k}}<\Lambda.$$
\end{enumerate}
Then either $f_{k}$ converges smoothly to a conformal minimal branched immersion from $D$ into $N$ with $\|\omega_{k}\|_{L^{\infty}(D_r(0))}<C(r)$, or $f_{k}$ converges to a point. For the case $f_{k}$ converges to a point, there exists $\lambda_{k}\rightarrow0$, such that $f_{k}'=\frac{f_{k}-f_{k}(0)}{\lambda_{k}}$ converge smoothly to a minimal branched immersion of $D$ in $\mathbb{R}^n$ with $\|\omega_{k}'\|_{L^{\infty}(D_r(0))}<C(r)$.
\end{pro}
\proof
It is easy to see that $f_{k}$ is a harmonic map from $(D,g_{f_{k}})$ to $N$ (see \cite{Sacks-Uhlenbeck1}). Hence,  $f_k$ satisfies the following Euler-Lagrange equation
$$-\Delta_{g_{k}}f_{k}=A_{N, \mathbb{R}^{n}}(f_{k})(\nabla_{g_{k}}f_{k},\nabla_{g_{k}}f_{k}).$$
By the conformal invariance of harmonic map equation, we have
\begin{equation}\label{harmonic1}
-\Delta f_{k}=A_{N, \mathbb{R}^{n}}(f_{k})(\nabla f_{k},\nabla f_{k}).
\end{equation}

We need to consider the following two cases. If $\{f_{k}\}$ does not converge to a point, then by Theorem \ref{thm1} we know that $f_{k}$ converge weakly in $W^{2,2}_{loc}$ and $$\|\omega_{k}\|_{L^{\infty}(D(r))}< C(r).$$ It follows that $$\|\nabla f_{k}\|_{L^{\infty}(D_r(0))}=\||z|^{m}e^{\omega_{k}}\|_{L^{\infty}(D_{r}(0))}< C(r).$$
Hence, from \eqref{harmonic1} we know that $f_{k}$ converge smoothly.

If $\{f_{k}\}$ converges to a point, Theorem \ref{thm2} tells us that there exists $\lambda_{k}\rightarrow 0$ such that $f_{k}'$ converge in $W^{2,2}_{loc}$. Because $f_{k}'$ satisfies the following equation
\begin{equation}\label{harmonic2}
-\Delta f_{k}'=\lambda_{k}A(\lambda_{k}f_{k}'+f_{k}(0))(\nabla f_{k}',\nabla f_{k}'),
\end{equation}
it follows that $f_{k}'$ converges smoothly to a minimal branched immersion.
\endproof
\medskip

\section{Bubble tree and total curvature identity }\label{bubble}
In this section, we will first give a bubble tree construction of $f_k$, and then provide the proof of Theorem \ref{identity}.
\subsection{Convergence of integral of Gauss curvature}

\begin{lem}
Let $\{f_k\}$ be a sequence of branched conformal  immersions from $D$ in $\R^{n}$ with $\sup_k\int_D|A_{f_k}|^{2}<\gamma_n$, which converges weakly in $W^{2,2}_{loc}(D)$ to an $f_{0}\in W^{2,2}_{b,c}(D)$, where $f_{0}$ is not a point map. Then we have
\begin{enumerate}
\item[(1)] for any $t\in (0,1)$
$$\lim_{k\rightarrow +\infty}\int_{D_{t}}K_{f_{k}}d\mu_{f_{k}}=\int_{D_{t}}K_{f_{0}}d\mu_{f_{0}}.$$
\item[(2)] for any $t\in(0,1)$,
$$\lim_{k\rightarrow +\infty}\int_{\partial D_{t}}\frac{\partial u_{k}}{\partial r}=\int_{\partial D_{t}}\frac{\partial u_{0}}{\partial r};$$
\end{enumerate}
\end{lem}

\proof
Let $g_{f_{k}}=e^{2u_{k}}g_{euc}$. Then, by Theorem \ref{remove} we know that $u_{k}$ satisfies
\begin{equation}\nonumber
-\Delta u_{k}=-2m\pi \delta_{0}+K_{f_{k}}e^{2u_{k}}.
\end{equation}
Thus, $\omega_{k}$ satisfies
\begin{equation}\label{Guo}
-\Delta \omega_{k}=K_{f_{k}}e^{2u_{k}},
\end{equation}
for $k=0,1,\cdots $.
By Theorem \ref{thm1}, Remark \ref{userem} and Remark \ref{rmk}, it is easy to see that $\omega_k$ converges to $\omega_0$ weakly in $W^{1,2}_{loc}(D)$. Thus for any $\phi\in C^{\infty}_{0}(D)$, we have
$$\int_{D}\nabla \omega_k \nabla \phi\rightarrow \int_{D}\nabla \omega_0 \nabla \phi. $$
It follow from \eqref{Guo} that
$$\int_{D}K_{f_k}e^{2u_k}\phi\rightarrow \int _{D} K_{f_0}e^{2{u_0}}\phi.$$
In other worlds, $K_{f_{k}}e^{2u_{k}}$ converges to $K_{f_{0}}e^{2u_{0}}$ in distribution sense. Then, $(1)$ can be directly deduced by Theorem 1 in \cite{evans} on page 54.

Now, we return to the proof of (2).  Since
$$\lim_{k\rightarrow +\infty}\int_{D_{t}}K_{f_{k}}d\mu_{f_{k}}=\int_{D_{t}}K_{f_{0}}d\mu_{f_{0}}.$$
for any $t\in(0,1)$, it is easy to see from \eqref{Guo} that
$$\lim_{k\rightarrow +\infty}\int_{\partial D_{t}}\frac{\partial \omega_{k}}{\partial r}=\int_{\partial D_{t}}\frac{\partial \omega_{0}}{\partial r},$$
which implies
$$\lim_{k\rightarrow +\infty}\int_{\partial D_{t}}\frac{\partial u_{k}}{\partial r}=\int_{\partial D_{t}}\frac{\partial u_{0}}{\partial r}.$$
\endproof
\subsection{First Bubble}\label{bubbletree}
We now turn our attention to giving the precise bubble tree construction.
Let $f_{k}: D\rightarrow \R^{n}$ be a sequence of conformal branched immersions with branching order $m$ at $0$ which satisfies
\begin{enumerate}
\item[(1)] $$\sup_{k}\int_{D}(1+|A_{f_{k}}|^{2})d\mu_{f_{k}}<\Lambda;$$
\item[(2)] $f_{k}(D)$  can be extended to a closed branched immersed surface $f_{k}:\Sigma\rightarrow \R^{n}$ with
$$\int_{\Sigma}|A_{f_{k}\comp \Sigma, \R^{n}}|^{2}d\mu_{f_{k}}<\lambda.$$
\end{enumerate}
Denote
$$\mathcal{S}(\{f_{k}\}):\ = \left\{p\in D:\lim_{t\rightarrow 0}\liminf_{k\rightarrow\infty}\int_{D_{t}(p)}|A_{f_{k}}|^{2}d\mu_{f_{k}}\geq\gamma_{n}\right\}.$$
For simplicity, we assume $\mathcal{S}(\{f_{k}\})=\{0\}$. Let $r_{k}$ be the smallest $r$, s.t.
$$\sup_{x\in\overline{D_{\frac{1}{2}}(0)}}\int_{D_{r}(x)}|A_{f_{k}}|^{2}d\mu_{f_{k}}=\gamma_{n}-\epsilon_{0},$$
here $\epsilon_{0}$ is any fixed sufficiently small positive constant. Let $z_{k}\in\overline{D_{\frac{1}{2}}}$, s.t.
$$\int_{D_{r_{k}}(z_{k})}|A_{f_{k}}|^{2}d\mu_{f_{k}}=\gamma_{n}-\epsilon_{0}.$$
We have the following two claims: $1)\s r_k\rightarrow0\s\s\mbox{and}\s\s 2)\s z_{k}\rightarrow 0$.\medskip

If $1)$ is false, assume $r_{k}\rightarrow r_0>0$. By the definition of $r_{k}$ it is easy to see that
$$\int_{D_{\frac{r_{0}}{2}}(0)}|A_{f_{k}}|^{2}d\mu_{f_{k}}<\gamma_{n}-\epsilon_{0},$$
when $k$ sufficiently large. Hence $0\notin\mathcal{S}(\{f_{k}\})$, which contradicts the fact $\mathcal{S}(\{f_{k}\})=\{0\}$.

If $2)$ is not true, then the limit of $z_{k}$ is another concentration point, which contradicts to our assumptions that $\mathcal{S}(\{f_{k}\})=0$.\medskip

Now let $f'_{k}(x)=f_{k}(r_{k}x+z_{k})-f_{k}(z_{k})$, then we have
$$\int_{D_{r_{k}}(z_{k})}|A_{f_{k}}|^{2}d\mu_{f_{k}}=\int_{D}|A_{f'_{k}}|^{2}d\mu_{f'_{k}}=\gamma_{n}-\epsilon_{0}.$$
Note that for any fixed $R>0$, $x\in D_{R}(0)$,  there holds
$$\int_{D_{1}(x)}|A_{f_{k}}|^{2}d\mu_{f_{k}}\leq \gamma_{n}-\epsilon_{0},$$
when $k$ sufficiently large . Thus  $\mathcal{S}(\{f'_{k}\})\cap D_{R}(0)=\emptyset$. By Theorem \ref{thm1} and Theorem \ref{thm2}, we know the generalized limit $f^{1}$ of $f'_{k}$ is a conformal branched immersion of $\mathbb{C}$ with
$$\int_{D}|A_{f^{1}}|^{2}d\mu_{f^{1}}\leq \gamma_{n}-\epsilon_{0},$$
Usually, we call $f^{1}$ the first bubble. We call the bubble is trivial if $ \int_{\C} K_{f^{1}}d\mu_{f^{1}}=0$.

By the above argument,  the blow up behavior of $f_k$ in $D\backslash D_{\delta}(0)$ where $\delta$ can be sufficiently small and $D_{Rr_k}(z_k)$ where $R$ is any positive number is already clear. In order to study the blow up behavior of $f_k$ in $D_{\delta}\backslash D_{Rr_k}(z_k)$, as harmonic maps, it will be convenient  to make a conformal transformation which converts $D\backslash D_{r_k}(z_k)$ to a long cylinder. More precisely, let $(r,\theta)$ be the polar coordinate centered at $z_k$ and set $T_k=-\ln r_k$. Let $\phi_k: S^{1}\times [0,T_k]\rightarrow \R^{2}$ be the conformal mapping given by $\phi_k(\theta,t)=z_k+(e^{-t},\theta)$. Then
$$\phi_k^{*}(dx^{1}\otimes dx^{1}+dx^{2}\otimes dx^{2})=e^{-2t}(dt^{2}+d\theta^{2}).$$
Thus $f_k\comp \phi_k$ is a sequence of conformal immersions of $S^{1}\times [0,T_k]$ in $\R^{n}$. We also denote $f_k\comp \phi_k$ by $f_k$ for simplicity of notations.

By the construction of the first bubble, it is easy to see that $\frac{f_k(\theta, t)-f_k(0,1)}{\diam{(f_k}(S^{1}\times \{1\}))}$ convergence weakly in $W^{2,2}_{loc}(S^{1}\times [0,\infty])$ to some $f^{+}\in W^{2,2}_{conf, loc}(S^{1}\times [0,\infty))$ and $\frac{f_k(\theta, t+T_k)-f_k(0,T_k-1)}{\diam{(f_k}(S^{1}\times \{T_k-1\}))}$ convergence weakly in $W^{2,2}_{loc}(S^{1}\times (-\infty,0])$ to some $f^{-}\in W^{2,2}_{conf, loc}(S^{1}\times (-\infty,0])$. 	We call $f^{-},f^{+}$ the top and the bottom of $f_k$ on $S^{1}\times [0,T_k]$ respectively. Additionally, we say $f_k$ has concentration on $S^{1}\times[0,T_k]$, if we can find $(\theta_k,t_k)\in S^{1}\times [0,T_k]$, such that
$$\lim_{r\rightarrow 0}\lim_{k\rightarrow \infty}\int_{B_{r}(\theta_k,t_k)}|A_{f_k}|^{2}d\mu_{f_k}\geq \gamma_n.$$

\subsection{When $f_k$ has no concentration.} In this subsection, we will give a bubble tree construction and prove the "no-neck" property when $f_k$ has no concentration on $[0,T_k]$.
\begin{lem}\label{cb1}
If $f_k$ has no concentration points in $S^{1}\times[0,T_k]$, then after passing to a subsequence, we may find $0=t_k^{0}<t_k^{1}<\cdots<t_k^{d}<t_k^{d+1}=T_k$ where $d<\frac{\Lambda}{4\pi}$such that
\begin{equation}\label{neck1}
t_k^{i}-t_k^{i-1}\rightarrow\infty,\s \mbox{for} \s i=1,\cdots, d+1;
\end{equation}
\begin{equation}\label{neck2}
\int_{S^{1}\times \{t_k^{i}\}}\frac{\partial u_k}{\partial t}=2m_i\pi+\pi,\s \mbox{where} \s m_i \in \Z \s \mbox{and} \s i=1,\cdots,d;
\end{equation}
and, for any $l_k\in(t_k^{i},t_k^{i+1})$ where $i=0,\cdots, d$ with $l_k-t_k^{i}\rightarrow+\infty$ and $t_k^{i+1}-l_k\rightarrow +\infty$, there holds true
\begin{equation}\label{neck3}
\int_{S^{1}\times\{l_k\}}\frac{\partial u_k}{\partial t}\neq 2m\pi+\pi, \ \ \mbox{where}\  m\in\Z.
\end{equation}
\end{lem}
\begin{proof}
Firstly, noting that for any $0\leq r\leq s\leq T_k$,
\begin{equation}\nonumber
\left|\int_{S^{1}\times\{r\}}\frac{\partial u_k}{\partial t}-\int_{S^{1}\times\{s\}}\frac{\partial u_k}{\partial t}\right|=\left|\int_{S^{1}\times [r,s]}K_{f_k}d\mu_{f_k}\right|\leq\int_{S^{1}\times[r,s]}|A_{f_k}|^{2}d\mu_{f_k}\end{equation}

Since $\int |A_{f_k}|^{2}d\mu_{f_k}$ is uniformly bounded, we know $\int_{S^{1}\times\{r\}}\frac{\partial u_k}{\partial t}$ is a continuous function with respect to $r$.

Suppose $\Lambda<4m\pi,$
where $m$ is a positive integer. We prove the lemma by induction on $m$.

When $m=1$,  we know $d=0$ and \eqref{neck2} is vacuous. We only need to prove \eqref{neck3}. If \eqref{neck3} is false, then there exists a sequence $l_k\in[0,T_k]$ satisfying $l_k\rightarrow+\infty$ and $T_k-l_k\rightarrow +\infty$ such that
$$\int_{S^{1}\times\{l_k\}}\frac{\partial u_k}{\partial t}= 2m\pi+\pi, \s \mbox{for some}\  m\in\Z.$$
Combining the fact that $f_k$ has no concentration with Theorem \ref{thm2}, we know there exists some $f^{1}\in W^{2,2}_{conf,loc}(S^{1}\times \R)$ such that
$$f_k'(t)=\frac{f_k(\theta,l_k+t)-f_k(0,l_k)}{\diam (f_{k}(S^{1}\times\{l_k\}))}\rightarrow f^{1} \s \mbox{weakly in} \s W^{2,2}_{loc}(S^{1}\times \R).$$
Moreover, if we denote conformal factor of $f^{1}$ by $e^{2u_{f^{1}}}$, then
$$\int_{S^{1}\times \{0\}}\frac{\partial u_{f^{1}}}{\partial t}=2m\pi+\pi.$$
By Theorem \ref{infty} and Remark \ref{rmk}, we know
$$\lim_{t\rightarrow+\infty}\int_{S^{1}\times \{t\}}\frac{\partial u_{f^{1}}}{\partial t}=2m_1\pi,\s
\mbox{and}\s
\lim_{t\rightarrow-\infty}\int_{S^{1}\times \{t\}}\frac{\partial u_{f^{1}}}{\partial t}=2m_2\pi,$$
where $m_1, m_2\in\Z$. Thus
$$\left|\int_{S^{1}\times(-\infty,0]}K_{f^{1}}d\mu_{f^{1}}\right|\geq\pi,\s
\mbox{and}\s
\left|\int_{S^{1}\times[0,+\infty)}K_{f^{1}}d\mu_{f^{1}}\right|\geq\pi.$$
It follows from the above that
$$\int_{S^{1}\times(-\infty,+\infty)}|A_{f^{1}}|^{2}d\mu_{f^{1}}\geq4\pi,$$
which contracts that $\Lambda<4\pi$. Thus \eqref{neck3} holds true when $m=1$.

Assuming the lemma is true for $m-1$, we will prove that it also holds for $m$. First of all, if \eqref{neck3} holds for any $l_k\in[0,T_k]$ with $l_k\rightarrow+\infty$ and $T_k-l_k\rightarrow +\infty$, then the lemma \ref{cb1} follows obviously. Otherwise, there exists a sequence $l_k\in[0,T_k]$ satisfying $l_k\rightarrow+\infty$ and $T_k-l_k\rightarrow +\infty$ such that
$$\int_{S^{1}\times\{l_k\}}\frac{\partial u_k}{\partial t}= 2m\pi+\pi, \s \mbox{for some}\  m\in\Z.$$
Adapting the same argument as above, we know there exists an $f^{1}\in W^{2,2}_{conf,loc} (S^{1}\times \R)$ which is the generalized limit of $f_k(\theta,l_k+t)$ and satisfies
$$\int_{S^{1}\times \R}|A_{f^{1}}|^{2}d\mu_{f^{1}}\geq4\pi.$$
Thus we can find $T$, such that
$$\int_{S^{1}\times[0,l_k-T]}|A_{f_k}|^{2}d\mu_{f_k}<4(m-1)\pi,\s\mbox{and} \s\int_{S^{1}\times[l_k+T,T_k]}|A_{f_k}|^{2}d\mu_{f_k}<4(m-1)\pi.$$
Using the induction hypothesis on $[0,l_k-T]$ and $[l_k+T,T_k]$, we can find
$$0=\overline{t}_{k}^{0}<\overline{t}_{k}^{1}<\cdots<\overline{t}_k^{\overline{l}}<\overline{t}_k^{\overline{l}+1}=l_k-T\s\mbox{and} \s l_k+T=\hat{t}_{k}^{0}<\hat{t}_{k}^{1}<\cdots<\hat{t}_k^{\hat{l}}<\hat{t}_k^{\hat{l}+1}=T_k$$
which satisfy \eqref{neck1}, \eqref{neck2}, and \eqref{neck3}.

Set
$$
t_k^{i}=\left\{\begin{aligned}
\overline{t}_k^{\overline{l}}\s\s    & \mbox{when} \s i\leq\overline{l},\\
l_k \s\s    & \mbox{when} \s i=\overline{l}+1,\\
\hat{t}_k^{i-\overline{l}-1}\s\s   &\mbox{otherwise}.
\end{aligned}\right.
$$
It is easy to check that the above partition satisfies \eqref{neck1}, \eqref{neck2}, and \eqref{neck3}. Thus we complete the proof.
\end{proof}

Now, let
$$f_k^{i}(\theta,t)=\frac{f_k(\theta, t_k^{i}+t)-f_k(0,t_k^{i})}{\diam(f_k(S^{1}\times\{t_{k}^{i}\}))}.$$
Since $f_k$ has no concentration, we know $f_k^{i}$ converges weakly in $W^{2,2}_{loc}(S^{1}\times \R)$ to some $f^{i}\in W^{2,2}_{conf,loc}(S^{1}\times \R,\R^{n})$. We call such $f^{i}$ are bubbles of $f_k$ (note that $f^{i}$ may be trivial, here trivial means $\int K_{f^{i}} d\mu_{f^{i}}=0$). We have the following proposition.
\begin{pro}
If $f_k$ has no concentration  on $S^{1}\times[0,T_k]$, then we have
$$
\lim_{k\rightarrow \infty}\int_{S^{1}\times[0,T_k]}K_{f_k}d\mu_{f_k}=\int_{S^{1}\times[0,+\infty)}K_{f^{+}}d\mu_{f^{+}}+\int_{S^{1}\times(-\infty,0]}K_{f^{-}}d\mu_{f^{-}}+\sum_{i=1}^{l}\int_{S^{1}\times \R}K_{f^{i}}d\mu_{f^{i}}.
$$
\end{pro}
\begin{proof}
It is enough for us to prove the integral of Gauss curvature on the neck domain connecting $f^{i}$ and $f^{i+1}$ equals to $0$, i.e.,
$$\lim_{T\rightarrow \infty}\lim_{k\rightarrow \infty}\int_{S^{1}\times [t_k^{i}+T,t_k^{i+1}-T]}K_{f_k}d\mu_{f_k}=0.$$
By Theorem \ref{infty} and Remark \ref{rmk}, there holds
$$\lim_{T\rightarrow +\infty}\lim_{k\rightarrow +\infty}\int_{S^{1}\times\{t_k^{i}+T\}}\frac{\partial u_k}{\partial t}=\lim_{s\rightarrow +\infty}\int_{S^{1}\times\{s\}}\frac{\partial u_{f^{i}}}{\partial t}= 2m_1\pi \s\mbox{for some} \s m_1\in\Z,$$ and
$$\lim_{T\rightarrow +\infty}\lim_{k\rightarrow \infty}\int_{S^{1}\times\{t_k^{i+1}-T\}}\frac{\partial u_k}{\partial t}=\lim_{s\rightarrow -\infty}\int_{S^{1}\times\{s\}}\frac{\partial u_{f^{i+1}}}{\partial t}= 2m_2\pi  \s\mbox{for some} \s m_2\in\Z.$$
We claim that $m_1=m_2$. If $m_1\neq m_2$, by the fact that $\int_{S^{1}\times\{s\}}\frac{\partial u_k}{\partial t}$ is continuous function, it is easy to see that there exists a sequence $l_k$ with $l_k-t_k^{i}\rightarrow +\infty$ and $t_k^{i+1}-l_k\rightarrow +\infty$ satisfying
$$\int_{S^{1}\times \{l_k\}}\frac{\partial u_k}{\partial t}=2m\pi+\pi, \s \mbox{where}\s m=\min\{m_1.m_2\},$$
which contradicts to \eqref{neck3}. Thus
\begin{equation}
\lim_{T\rightarrow \infty}\lim_{k\rightarrow \infty}\int_{S^{1}\times [t_k^{i}+T,t_k^{i+1}-T]}K_{f_k}d\mu_{f_k}=2m_1\pi-2m_2\pi=0.\end{equation}
From the above it is easy to see that the integral of Gauss curvature on the neck domain equals to $0$. Hence we complete the proof.
\end{proof}

\subsection{When $f_k$ has concentration} In this subsection, we will show that how to find out all the bubbles when $f_k$ has concentration on $[0,T_k]$. Using  a similar  induction argument as in the proof of Lemma \ref{cb1}, we have the following
\begin{lem}
There exists a subsequence of $f_k$ and $0=d_k^{0}<d_k^{1}<\cdots<d_k^{l}<d_k^{l+1}=T_k$ with $l<\Lambda/4\pi$, such that
\begin{eqnarray}
d_k^{i+1}-d_k^{i}\rightarrow +\infty,  \s\s i=1,\cdots, l+1;\label{neck4}\\
\int_{S^{1}\times(d_k^{i}-1,d_k^{i}+1)}|A_{f_k}|^{2}d\mu_{f_k}\geq 4\pi,\s\s i=1,\cdots, l;\label{neck5}\\
\lim_{T\rightarrow \infty}\liminf_{k\rightarrow +\infty}\sup_{t\in[d_k^{i}+T,d_k^{i+1}-T]}\int_{S^{1}\times (t-1,t+1)}|A_{f_k}|^{2}d\mu_{f_k}<4\pi, \s \s i=0,\cdots, l.\label{neck6}
\end{eqnarray}
\end{lem}

From \eqref{neck6}, it is obviously that we can find $T'>0$ such that $f_k$ has no concentration on $S^{1}\times[d_k^{i}+T',d_k^{i+1}-T']$. By Lemma \ref{cb1}, the blow up behavior of $f_k$ and convergence behavior of integral of Gauss curvature at such components are clear.

Now, let $f_k^{i}(\theta,t)=f_k(\theta, d_k^{i}+t)$. If we denote the concentration points of $f_k^{i}$ by $\mathcal{S}(\{f_k^{i}\})$, then, we know there exists $\lambda_k=1$ or $\lambda_k\rightarrow 0$ such that $\frac{f_k^{i}}{\lambda_k}$ converges weakly in $W^{2,2}_{loc}(S^{1}\times\R\backslash \mathcal{S}(\{f_k^{i}\})$ to some $\hat{f}^{i}$ which is a branched conformal immersion from $S^{1}\times\R$ to $\R^{n}$. Note that the top of $f_k$ on $S^{1}\times[d_k^{i}+T', d_k^{i+1}-T']$ is just a part of generalized limit of $f_k^{i+1}$ and the bottom of $f_k$ on $S^{1}\times[d_k^{i}+T', d_k^{i+1}-T']$ is just a part generalized limit of $f_k^{i}$. Up to now, from the construction, in deed we know the bubbles come from two different ways. The one is the bubbles constructed by Lemma \ref{cb1} on the component $S^{1}\times [d_k^{i}+T',d_k^{i+1}-T']$; Another is $\{\hat{f}^{i}\}$ which is the generalized limit of $\{f_k^{i}\}$. We call all these bubbles are bubbles of first level of bubble tree. By Lemma \ref{cb1} again, we have
\begin{align*}\lim_{k\rightarrow \infty}\int_{S^{1}\times [0,T_k]}K_{f_k}d\mu_{f_k}=&\sum_{p\in \mathcal{S}(\{f_k^{i}\})}\lim_{r\rightarrow0}\lim_{k\rightarrow \infty}\int_{B_{r}(p)}K_{f_{k}^{i}}d\mu_{f_k^{i}}+\sum_{j}\int_{S^{1}\times\R}K_{f^{i}}d\mu_{f^{i}}\\
&+\int_{S^{1}\times[0,+\infty)}K_{f^{+}}d\mu_{f^{+}}+\int_{S^{1}\times(-\infty,0]}K_{f^{-}}d\mu_{f^{-}},
\end{align*}
where $f^{i}$ are all the bubbles of the first level of bubble tree.

For each $p\in \mathcal{S}(f_k^{i})$, we take a small $r$ such that $B_r(p)\subset S^{1}\times\R$ contains only one blowup point. We can use the same argument as in the construction of first level bubbles of $f_k$ to construct first level bubbles of $f_k^{i}$. We call them the second level of bubble tree.

Step by step, we can construct the third, forth, ... level of the bubble tree. Obviously such construction will stop after finite many steps. Finally, by deleting trivial bubbles, we get finite non-trivial bubbles $f^{1},\cdots, f^{m}$.

By Theorem \ref{infty} and Theorem \ref{remove}, we can regard $f^{i}$ as a conformal branched immersion from $\C$ to $\R^{n}$. Combining all the above arguments, we get the following theorem:
\begin{thm}\label{disk}
Let $f_{k}: D\rightarrow \R^{n}$ be a sequence of conformal branched immersions with branching order $m$ at $0$ which satisfies
\begin{enumerate}
\item[(1)] $$\sup_{k}\int_{D}(1+|A_{f_{k}}|^{2})d\mu_{f_{k}}<\infty;$$
\item[(2)] $f_{k}(D)$  can be extended to a closed branched immersed surface $f_{k}:\Sigma\rightarrow \R^{n}$ with
$$\int_{\Sigma_{k}}|A_{f_{k}\comp \Sigma, \R^{n}}|^{2}d\mu_{f_{k}}<\Lambda.$$
\end{enumerate}
Then for any $r<1$, there holds
$$\lim_{k\rightarrow\infty}\int_{D_{r}}K_{f_{k}}d\mu_{f_{k}}=\int_{D_{r}}K_{f_{0}}d\mu_{f_{0}}+\sum_{i=1}^{m}\int_{\mathbb{C}}K_{f^{i}}d\mu_{f^{i}},$$
where $f^{1}, ...\ , f^{m}$ are all the non-trivial bubbles.
\end{thm}

Now, we are ready to prove the main Theorem \ref{identity}.
\subsection{Proof of Theorem \ref{identity}:}
By Theorem \ref{gauss}, we know that for any $k$, there holds
$$\int_{\Sigma}K_{f_k}d\mu_{f_k}=2\pi\chi(\Sigma)+2\pi b,$$
here $b$ is the sum of the branching order of the branch points. Since  $\mathcal{S}(\{f_{k}\})$ is a finite discrete set, assume $\mathcal{S}(\{f_{k}\})=\{p_{1},\ p_{2},..., p_{m}\}$. We can choose $\delta$, such that $B_{\delta}(p_{i})\cap B_{\delta}(p_{j})=\emptyset$. Taking an isothermal coordinate system around $p_{i}$, the result can be directly deduced from Theorem \ref{disk}.

\section{Proof of theorem \ref{identity2}}
Before giving the proof, we first state the following theorem which gives an existence and compactness of conformal parameter under metric convergence.
\begin{thm}[\cite{DeTurck-Kazadan}]\label{con-metric}
Let $h_{k}$, $h_{0}$ be smooth Riemann metrics on a surface $M$, such that $h_{k}\rightarrow h_{0}$ in $C^{s,\alpha}(M)$, where $s\in \mathbb{N}$, $\alpha\in(0,l)$. Then for each point $z\in M$, there exists neighborhoods $U_{k}$, $U_{0}$ of $z$ and smooth conformal diffeomorphisms $\vartheta_{k}: D\rightarrow U_{k}$ and $\vartheta_{0}: D\rightarrow U_{0}$, such that $\vartheta_{k}\rightarrow \vartheta_{0}$ in $C^{s+1,\alpha}(\overline{D}, M)$.
\end{thm}
\begin{rem}\label{also}
By Theorem \ref{con-metric}, it is easy to see from the proof of Theorem \ref{thm1} in \cite{lamm-nguyen} that the statements of Theorem \ref{thm1} still hold for branched immersions $f_{k}\in W^{2,2}_{b,c}(D,\R^{n})$ with the induced metric $g_{k}=|z|^{2m}e^{2\omega_{k}}g_{0,k}$ if $(g_{0,k})_{ij}$ converges to $\delta_{ij}$ smoothly on $\overline{D}$.
\end{rem}

\subsection{Bubbles from long cylinders}\label{collar}
In this subsection, we develop the blowup analysis of conformal immersions from long cylinders to deal with the degeneration of complex structure. In fact, we have ever considered it in the previous section. In order to emphasis its usefulness and importance in handling the degenerating case, we still write it down.

Let  $f_{k}:S^{1}\times[0,T_{k}]\rightarrow \mathbb{R}^{n}$ with $T_{k}\rightarrow +\infty$ be a sequence of conformal immersions which satisfies
\begin{enumerate}
\item[(1)] $f_{k}(S^{1}\times[0,T_{k}])$ can be extended to a closed surface with
$$\int_{\Sigma_{k}}|A_{\Sigma_{k}}|^{2}d\mu<\Lambda;$$
\item[(2)] $f_{k}(\theta,t)$ has no concentration points on $S^{1}\times[0,+\infty)$ and $f_{k}(\theta,T_{k}+t)$ has no concentration points on $S^{1}\times(-\infty,0]$. That is to say, if we denote $\lambda_{k}$ and $\tau_k$ by $\diam{f_k(S^{1}\times\{1\})}$ and $\diam{f_k(S^{1}\times\{T_k-1\})}$ respectively, then $\frac{f_{k}(\theta,t)-f_{k}(\pi,1)}{\lambda_{k}}$ and $\frac{f_{k}(\theta,t+T_{k})-f_{k}(\pi,T_{k}-1)}{\tau_{k}}$ converge weakly in $W^{2,2}_{loc}$ to some $f^{+}\in W^{2,2}_{conf}(S^{1}\times [0,+\infty))$ and $f^{-}\in W^{2,2}_{conf}(S^{1}\times (-\infty,0])$ respectively;
\item[(3)] $\mu(f_{k})<\Lambda.$
\end{enumerate}
By the arguments in the previous section, the following theorem holds obviously.
\begin{thm}\label{collar-idtt}
Assume $\{f_{k}\}$ satisfies the above conditions, and  $f^{1}, f^{2}, \cdot\cdot\cdot, f^{m}$ are $m$ non-trivial bubbles of $f_k$. Then we have
$$\lim_{k\rightarrow+\infty}\int_{S^{1}\times[0,T_{k}]}K_{f_{k}}=\int_{S^{1}\times[0,+\infty)}K_{f_{0}^{+}}+\int_{S^{1}\times(-\infty,0]}K_{f_{0}^{-}}+\sum_{i=1}^{m}\int_{\mathbb{C}}K_{f^{i}}.$$
\end{thm}\medskip

\subsection {\bf Proof of Theorem \ref{identity2}:} We now turn our attention to giving the proof of  Theorem \ref{identity2}. The proof of Theorem \ref{identity2} will consists the following three cases according to the genus of $\Sigma$.

\noindent {\bf Spherical case :} When $\Sigma$ is a sphere, since there is only one conformal structure on $S^{2}$, we may let $h_{k}\equiv h$. If the number of branch points of $f_k$ is less or equal to $3$, we can always find a M\"obius transformation $\phi_{k}$ which converts the branch points (or point) of $f_k$ to the fixed points (or point). Then $f_{k}'=f_k\comp \phi_k$ is branched conform immersion with the the same branch points . The desired result can be directly deduced by Theorem \ref{identity}.  Other case will be discussed blow by the moduli theory of  spaces with marked point.
 \medskip

\noindent {\bf Toric case:} If $f_{k}$ has no branch point on $\Sigma_{k}$. Suppose that $(\Sigma, g)$ is induced by lattice $\{1,a+b\mi\}$ in $\C$, where $-\frac{1}{2}<a\leq\frac{1}{2}$, $b>0$, $a^{2}+b^{2}\geq1$, and $a\geq0$ whenever $a^{2}+b^{2}=1$. Then the conformal map $f$ from $(\Sigma,g)$ into $\R^{n}$ can be composed with the projection $\C\rightarrow \Sigma$ to yield a conformal map $\widetilde{f}$ from $\C$ into $\R^{n}$ which satisfies
$$\widetilde{f}(z+\lambda)=\widetilde{f}(z),\s\s \mbox{for} \s\mbox{all} \ \ \lambda\in \Z\otimes\Z(a+b\mi).$$
Let $\Pi:\C\rightarrow S^{1}\times\R$ defined by $x+y\mi\rightarrow (2\pi x,2\pi y)$ be the conformal covering map, where $2\pi x$ and $2\pi (x+m)$ are the same point in $S^{1}$ for $m\in\N$. Then $(\Sigma, g)$ is conformal to $(S^{1}\times\R)/G$, where $G\cong \Z$ is the transformation group of $S^{1}\times\R$ generated by the mapping $(\theta,t)\rightarrow (\theta+2\pi a, t+2\pi b)$. Then $\widetilde{f}$ descents to a conformal map $f':S^{1}\times\R\rightarrow \R^{n}$, which satisfies $f'\comp \Pi=\widetilde{f}$.

Now we assume $(\Sigma, g_{k})=S^{1}\times\R/G_{k}$, where $G_{k}$ is generated by
$$(\theta,t)\rightarrow (\theta+\theta_{k},t+b_{k}),\ \ \mbox{where} \ \ b_{k}\geq\sqrt{\pi^{2}-\theta_{k}^{2}},\ \ \mbox{and} \ \theta_{k}\in[-\frac{\pi}{2},\frac{\pi}{2}].$$
In the moduli space $\mathcal{M}_{1}$ of genus 1 surface, $(\Sigma, g_{k})$ diverges if and only if $b_{k}\rightarrow+\infty$.

Then any $f_{k}\in W^{2,2}_{conf}((\Sigma, g_{k}),\R^{n})$ can be lifted to a conformal immersion $f_{k}':S^{1}\times\R\rightarrow\R^{n}$ with
$$f_{k}'(\theta,t)=f_{k}'(\theta+\theta_{k},t+b_{k}).$$
After translations, we may assume that $f_{k}'(\theta, t+\frac{b_{k}}{2})$ and $f_{k}'(\theta, t-\frac{b_{k}}{2})$ have no blow up points as $k\rightarrow\infty$.  We denote by $f_0^{-}$ and $f_{0}^{+}$ the generalized limit of
$f_{k}'(\theta, t+\frac{b_{k}}{2})$ and $f_{k}'(\theta, t-\frac{b_{k}}{2})$ respectively. Then $f_{k}'$ satisfies the condition in subsection \ref{collar}. By Theorem \ref{collar-idtt}, we have
\begin{thm}
$$\int_{S^{1}\times[0,+\infty)}K_{f_{0}^{+}}+\int_{S^{1}\times(-\infty,0]}K_{f_{0}^{-}}+\sum_{i}\int_{ \C}K_{f^{i}}d\mu_{f^{i}}=0,$$
where $f^{i} $ are all of the nontrivial bubbles of $f_k'$.
\end{thm}
If $f_{k}$ has branch points on $\Sigma_{k}$, we will consider the moduli space with marked points, see the details blow.
\medskip

\indent {\bf Hyperbolic case:}\label{hyper}
Let us firstly review the compactification of the moduli space of surface with marked points. We refer to \cite{Hummel, zhu} for more details on degenerating surface.
Let $\overline{\mathcal{M}}_{g,m}$ $(2g+m\geq 3)$ be the moduli space of compact Riemann surface of genus $g$ with $m$ marked points. Let $(\Sigma_{0}, x_{0,1},\cdots,x_{0,m})\in \overline{\mathcal{M}}_{g,m}$ with nodal point $\mathcal{N}=\{a_{1},\cdot\cdot\cdot, a_{m'}\}$. Geometrically, $\Sigma_{0}$ is obtained by pinching $m'$ non null homotopy closed curves which do not pass through any of $\{ x_{0,1},\cdots,x_{0,m}\}$ to points $a_{1},\cdots, a_{m}$, thus $\Sigma_{0}\backslash\mathcal{N}$ can be divided to finite components $\Sigma_{0}^{1},\cdots,\Sigma_{0}^{s}$. For each $\Sigma_{0}^{i}$, we can extend $\Sigma_{0}^{i}$ to a closed Riemann surface $\overline{\Sigma_{0}^{i}}$ by adding a point at each puncture. Moreover, the complex structure of $\Sigma_{0}^{i}$ can be extended smoothly to a complex structure of $\overline{\Sigma_{0}^{i}}$.

We say $h$ is a hyperbolic structure on $(\Sigma,x_{1},\cdots,x_{m})\in \mathcal{M}_{g,m}$ if $h$ is a complete metric on $\Sigma\backslash\{x_{1},\cdots,x_{m}\}$ with curvature $-1$ and finite volume. We say $h_{0}$ a hyperbolic structure on $ (\Sigma_{0},x_{0,1},\cdots,x_{0,m})$ if $h_{0}$ is a complete metric on
$$\Sigma_{0}\backslash\{a_{1},\cdots,a_{m'},x_{0,1},\cdots,x_{0,m}\}$$
with curvature $-1$ and finite volume. We denote
$$\Sigma_{0}(h_0,\delta)=\{p\in \Sigma_{0}\backslash\mathcal{N}: \injrad^{h_0}_{\Sigma_0\backslash\mathcal{N}}(p)<\delta\}\cup\mathcal{N},$$
and
$$\Sigma_{0}^{i,\delta}=\{x\in \Sigma_{0}^{i}: \injrad^{h_0}_{\Sigma_0\backslash\mathcal{N}}(x)>\delta\}.$$
We also denote $\Sigma_0(a_j,h_0,\delta)$ the connected component which contains $a_j$ of $\Sigma_0(h_0,\delta)$. 
For a surface $\Sigma$ with hyperbolic structure $h$ and with marked points $\{x_{1},\cdots,x_{m}\}$, we define $\Sigma^{*}=\Sigma\backslash \{x_{1},\cdots,x_{m}\}$, and $h^{*}$ to be a hyperbolic structure on $(\Sigma,x_{1},\cdots,x_{m})$ which is conformal to $h$ on $\Sigma^{*}$.

Let $\{(\Sigma_{k},x_{k,1},\cdots,x_{k,m})\}$ be a sequence of marked surfaces in $\mathcal{M}_{g,m}$ with hyperbolic structure $h_{k}$ and
$$ (\Sigma_{k},x_{k,1},\cdots,x_{k,m})\rightarrow (\Sigma_{0},x_{0,1},\cdots,x_{0,m})\ in\ \overline{\mathcal{M}}_{g,m} .$$

By proposition 5.1 in chapter 4 in \cite{Hummel}, there exits a maximal collection $\Gamma_{k}=\{\gamma_{k}^{1},...,\gamma_{k}^{m}\}$ of pairwise disjoint, simple closed curves in $\Sigma_{k}$ with $\ell_{k}^{j}=L(\gamma_{k}^{i})\rightarrow0$, such that after passing to a subsequence the following holds:
\begin{enumerate}
\item[(1)] There are maps $\varphi_{k}\in\ C^{0}(\Sigma_{k},\Sigma_{0}),$ such that $\varphi_{k}:\Sigma_{k}\backslash\Gamma_{k}\rightarrow\Sigma_{0}\backslash\mathcal{N}$ is a diffeomorphism and $\varphi_{k}(\gamma_{k}^{j})=a_{j}$ for $j=1,\cdots,m'$, and $\varphi_{k}(x_{k,j})=x_{0,j}$ for $j=1,\cdots,m$.
\item[(2)] For the inverse diffeomorphisms $\psi_{k}:\Sigma_{0}\backslash\mathcal{N}\rightarrow\Sigma_{k}\backslash\Gamma_{k}$, we have $\psi_{k}^{\ast}(h_{k})\rightarrow h_{0}$ in $C^{\infty}_{loc}(\Sigma_{0}^{*}\backslash\mathcal{N}),$ where $h_{0}$ determine a hyperbolic structure on $\Sigma_{0}\backslash\mathcal{N}.$
\item[(3)] Let $c_{k}$ be the complex structure over $\Sigma_{k}$, and $c_{0}$ be the complex structure on $\Sigma_{0}\backslash\mathcal{N}$. then $\psi_{k}^{\ast}(c_{k})\rightarrow c_{0}$ in $C^{\infty}_{loc}(\Sigma_{0}\backslash\mathcal{N})$.
\item[(4)] For each $\gamma_{k}^{j}$ as above, there is a collar $U_{k}^{j}$ containing $\gamma_{k}^{j}$, which is isometric to the cylinder
$$Q_{k}^{j}=S^{1}\times(-\frac{\pi^{2}}{\ell_{k}^{j}}+\tau_{k},\frac{\pi^{2}}{\ell_{k}^{j}}-\tau_{k})$$
with the metric
$$h_{k}^{j}=\left(\frac{\ell_{k}^{j}}{2\pi\cos(\frac{\ell_{k}^{j}}{2\pi}t)}\right)^{2}(dt^{2}+d\theta^{2}).$$
where $\tau_{k}=\frac{2\pi}{\ell_{k}^{j}}\arctan(\sinh(\frac{\ell_{k}^{j}}{2})).$ Moreover, for any $(\theta,t)\in S^{1}\times(-\frac{\pi^{2}}{\ell_{k}^{j}}+\tau_{k},\frac{\pi^{2}}{\ell_{k}^{j}}-\tau_{k})$, we have
$$\sinh(\injrad_{\Sigma_{k}}(\theta,t))\cos(\frac{\ell_{k}^{j}t}{2\pi})=\sinh(\frac{\ell_{k}^{j}}{2}).$$
\end{enumerate}

Let $\phi_{k}^{j}$ be the isometry between $Q_{k}^{j}$ and $U_{k}^{j}.$ Then $\varphi_{k}\circ\phi_{k}^{j}(\theta,\frac{\pi^{2}}{\ell_{k}^{j}}-\tau_{k}+t)\bigcup\varphi_{k}\circ\phi_{k}^{j}(\theta,-\frac{\pi^{2}}{\ell_{k}^{j}}+\tau_{k}+t)$ converges in $C^{\infty}_{loc}(S^{1}\times(-\infty,0)\bigcup S^{1}\times(0,+\infty))$ to an isometry from $S^{1}\times(-\infty,0)\bigcup S^{1}\times(0,+\infty)$ to $\Sigma_{0}(a_{j},h_{0},1)\backslash\{a_{j}\}$.

Now we let $\widetilde{f}_{k}=f_{k}\comp \psi_{k}$. Then by the above arguments we know
$$\widetilde{f}_{k}\in W^{2,2}_{conf,loc}(\Sigma_{0}\backslash \mathcal{N}, \psi_{k}^{*}h_{k}),$$
 and $\psi_{k}^{*}h_{k}\rightarrow h_{0}$ in $C^{\infty}_{loc}(\Sigma_{0}^{*}\backslash\mathcal{N})$. By Theorem \ref{con-metric} and Remark \ref{also}, we only need to consider the convergence near $\{x_{0,j},...,x_{0,m}\}$ and collars.

For any $x_{0,j}\in \{x_{0,j},...,x_{0,m}\}$, choose a complex coordinate $\{U,(x,y)\}$ on $\Sigma_{0}$ compatible with $c_{0}$, with $x_{0,j}=(0,0)$. Let $c_{k}'=\psi_{k}^{*}(c_{k})$. We set
$$e_{1}=\frac{\partial}{\partial x},\ e_{2}=c_{k}'(e_{1}),$$
and $h_{k}'$ is defined to be the metric on $U$
 $$h_{k}'(e_{1},e_{1})=h_{k}'(e_{2},e_{2})=1,\ h_{k}'(e_{1},e_{2})=0.$$
Then $h_{k}'$ is compatible with $c_{k}'$ and converge smoothly to a metric which is compatible with $c_{0}$ in $U$. From this we can regard  $\widetilde{f}_{k}$ as a conformal branched immersion from $(\bigcup\Sigma_{0}^{i},\psi_{k}^{*}h_{k})$ with the same branch points $\{x_{0,j},...,x_{0,m}\}$. Now we let $f_{0}^{i}$ be the generalized limit of $\widetilde{f}_{k}|_{\Sigma_{0}^{i}\backslash\cup_{j=1}^{m'}\Sigma_0(a_j,h_0,\delta)}$. As $\delta\rightarrow0$, $f_{0}^{i}$ can be considered as a conformal branched immersion from $\Sigma_{0}^{i}$ to $\mathbb{R}^{n}$. Thus we have
\begin{equation}\label{last1}
\lim_{\delta\rightarrow 0}\lim_{k\rightarrow\infty}\int_{\varphi_{k}^{-1}(\Sigma_{0}^{i}\backslash\cup_{j=1}^{m'}\Sigma_0(a_j,h_0,\delta))}K_{f_k}d\mu_{f_{k}}
=\int_{{\Sigma_{0}^{i}}}K_{f_{0}^{i}}d\mu_{f_{0}^{i}}+\sum_{m}\int_{\C}K_{f_i^{m}} d\mu_{f_i^{m}},
\end{equation}
where $f_0^{i}$ is the generalized limit of $\widetilde{f}_{k}$ on $\Sigma_{0}^{i}$ and  $\{f_{i}^{m}\}$ are bubbles of $\widetilde{f}_{k}$ on $\Sigma_{0}^{i}$.
Next, we consider the convergence of $f_{k}$ on the collars. We set $\hat{f}_{k}^{j}=f_{k}\comp \phi_{k}^{j}$ and $T_{k}^{j}=\frac{\pi^{2}}{l_{k}^{j}}-T$. We may choose $T$ to be sufficiently large such that $\hat{f}_{k}^{j}(T_{k}^{j}-t,\theta)$ and $\hat{f}_{k}^{j}(-T_{k}^{j}+t,\theta)$ have no blowup points. Thus $\hat{f}_{k}^{j}$ satisfies the condition in subsection \ref{collar}. Since
$$
\hat{f}_{k}^{j}=f_{k}\comp \phi_{k}^{j}=f_{k}\comp \psi_{k}\comp (\varphi_{k}\comp \phi_{k}^{j})=\widetilde{f}_{k}\comp (\varphi_{k}\comp \phi_{k}^{j}),
$$
we conclude from \eqref{last1} that  the convergence of $\hat{f}_{k}^{j}(T_{k}^{j}-t,\theta)$ and $\hat{f}_{k}^{j}(-T_{k}^{j}+t,\theta)$ has been handled.  Note that
\begin{equation}\label{last2}
\lim_{\delta\rightarrow 0}\lim_{k\rightarrow\infty}\int_{\varphi_{k}^{-1}(\Sigma_{0}(a_j,h_0,\delta))}K_{f_k}d\mu_{f_{k}}
=\lim_{T\rightarrow +\infty}\lim_{k\rightarrow +\infty}\int_{S^{1}\times [-T_k^{j}+T,T_k^{j}-T]}K_{\hat{f}_{k}^{j}}d\mu_{\hat{f}_{k}^{j}},
\end{equation}
we have
\begin{equation}\label{last3}\lim_{\delta\rightarrow 0}\lim_{k\rightarrow\infty}\int_{\varphi_{k}^{-1}(\Sigma_{0}(a_j,h_0,\delta))}K_{f_k}d\mu_{f_{k}}=\sum_{i'}\int_{\C}K_{f^{i'}},
\end{equation}
where $f^{i'}$ are all of the nontrivial bubbles of $\hat{f}_{k}^{j}$. Combining \eqref{last1} and \eqref{last3}, we obtain the desired consequence.

\medskip

\noindent {\bf Acknowledgment}:
The author would like to thank his supervisors Professor Yuxiang Li and Professor Youde Wang for their encouragement and inspiring advice. In particular, without the beneficial discussion with Professor Yuxiang Li and his help, this paper would not have been completed. The author is very grateful to the  referees for valuable suggestions and insightful comments. We want to thank Professor Lamm and Professor Nguyen for sincerely reminding me that Theorem 1.4 in the first manuscript version of this paper has been proved in their paper \cite{lamm-nguyen}. The work is supported by National Natural Science Foundation of China ( Grant Nos. 11671015 and 11731001).


{}

\vspace{1.0cm}

Guodong Wei

{\small\it Academy of Mathematics and Systems Sciences, Chinese Academy of Sciences, Beijing 100190,  P.R. China.}

{\small\it School of Mathematical Sciences, University of Chinese Academy of Sciences, Beijing 100049, China
}

{\small\it Email: weiguodong@amss.ac.cn}


\begin{thebibliography}{2}
\bibitem{bryant} R. L. Bryant; Conformal and minimal immersions of compact surfaces into the 4-sphere, {J. Differential Geom.} {\bf 17} (1982), no. 3, 455-473.

\bibitem{Chen} B. Y. Chen; Some conformal invariants of submanifolds and their applications, {Boll. Un. Mat. Ital. (4)}  {\bf 10} (1974), 380--385.

\bibitem{Chen-Li}J. Chen, Y. Li; Bubble tree of branched conformal immersions and applications to the Willmore functional, {Amer. J. Math.} {\bf 136} (2014), 1107--1154.

\bibitem{Coifman} R. Coifman, P. Lions, Y. Meyer, S. Semmes; Compensated compactness and {H}ardy spaces. {J. Math. Pures Appl. (9)}  {\bf 72}  (1993),  247--286.

\bibitem{DeTurck-Kazadan} D. DeTurck, J. Kazdan; Some regularity theorems in Riemann geometry. {Ann. Sci. \'{E}cole Norm. Sup. (4)}  {\bf 14}  (1981),  249--260.

\bibitem{Ding-Tian} W. Ding, G. Tian; Energy identity for a class of approximate
harmonic maps from surfaces, {Comm. Anal. Geom.} {\bf 3} (1995), 543-554.

\bibitem{Eschenburg-Tribuzy} J. Eschenburg, R. Tribuzy; Branch points of conformal mappings of surfaces, {Math. Ann.} {\bf 279} (1988), 621-633.

\bibitem{evans} L. C. Evans, R. F. Gariepy; Measure theory and fine properties of functions, {Studies in Advanced Mathematics.} CRC Press, Boca Raton, FL, (1992).

\bibitem{Helein} F. H{\'e}lein;  Harmonic maps, conservation laws and moving frames,  {Cambridge Tracts in Mathematics}, {\bf 150} {Second Edition},{Translated from the 1996 French original,
With a foreword by James Eells},{Cambridge University Press, Cambridge}(2002).

\bibitem{Hummel} C. Hummel; Gromov's compactness theorem for pseudo-holomorphic curves,  {Progress in Mathematics}, {\bf 151},{Birkh\"auser Verlag, Basel}(1997).

\bibitem{Kuwert-Li}E. Kuwert, Y. Li; $W^{2,2}$-conformal immersions of a closed Riemann
surface into $\mathbb {R}^n$, {Comm. Anal. Geom.} {\bf 20} (2012), 313--340.

\bibitem{Kuwert-scha}E. Kuwert,  R. Sch\"{a}tzle; Minimizers of the Willmore functional under fixed conformal class, {J. Differential Geom. } {\bf 93} (2013),  471--530.

\bibitem{Kuwert-Schatzle}E. Kuwert,  R. Sch\"{a}tzle; Removability of point singularities of Willmore surfaces, {Ann. of Math. (2)} {\bf 160}  (2004), 315--357.

\bibitem{lamm-nguyen} T. Lamm,  H. T. Nguyen; Quantitative rigidity results for conformal immersions. {Amer. J. Math.},  {\bf 136} (2014),  1409--1440.

\bibitem{lamm-scha} T. Lamm, R. M. Sch\"{a}tzle; Rigidity and non-rigidity results for conformal immersions, {Adv. Math.} {\bf 281} (2015), 1178--1201.

\bibitem{Li}Y. Li; Some remarks on {W}illmore surfaces embedded in $\mathbb{R}^3$, { J. Geom. Anal.} {\bf 26}  (2016),  2411--2424.

\bibitem{Li2} Y. Li; Weak limit of an immersed surface sequence with bounded Willmore functional, arxiv: 1109.1472.

\bibitem{Li-Wang}Y. Li and Y. Wang;  A weak energy identity and the length of necks for a sequence of Sacks-Uhlenbeck $\alpha$-harmonic maps, {Adv. Math.} {\bf 225} (2010), no. 3, 1134-1184.

\bibitem{moore} J. D. Moore; On conformal immersions of space forms, Global differential geometry and global analysis (Berlin, 1979), pp. 203--210, Lecture Notes in Math., 838, Springer, Berlin, 1981.

\bibitem{Muller- sverak} S. M{\"u}ller, V. {\v{S}}ver{\'a}k; On surfaces of finite total curvature, {J. Differential Geom.} {\bf 42} (1995), 229-258.

\bibitem{Riviere} T. Rivi\'ere; Lipschitz conformal immersions from degenerating Riemann surfaces with $L^{2}$-bounded second fundamental forms, {Adv. Calc. Var.} {\bf  6} (2013), 1--31.

\bibitem{riviere1} T. Rivi\'ere; Variational principles for immersed surfaces with $L^{2}$-bounded second fundamental form, {J. Reine Angew. Math.}  {\bf 695} (2014), 41--98.

\bibitem{riviere2} T. Rivi\'ere; Analysis aspects of Willmore surfaces, {Invent. Math.} {\bf 174} (2008),  1--45.

\bibitem{Sacks-Uhlenbeck1} J. Sacks, K. Uhlenbeck; The existence of minimal immersions of 2-spheres, {Ann. of Math.} {\bf 113} (1981), 1-24.

\bibitem{Simon} L. Simon; Existence of surfaces minimizing the {W}illmore functional,
 {Comm. Anal. Geom.} {\bf 1}  (1993),   281--326.

\bibitem{Willmore} T. J. Willmore;  Total curvature in {R}iemannian geometry,
{John Wiley \& Sons, New York}  (1892).

\bibitem{zhu} M. Zhu; Harmonic maps from degenerating Riemann surfaces,
 {Math. Z.} {\bf 264}  (2010), 63--85.

\end{thebibliography}
\end{document}